\documentclass[a4paper,10pt]{amsart}
\usepackage{amssymb,amscd,amsmath}

\theoremstyle{plain}
\newtheorem{thm}{Theorem}[section]
\newtheorem{prop}[thm]{Proposition}

\theoremstyle{definition}

\theoremstyle{remark}
\newtheorem{rem}{Remark}[section]

\newcommand{\forme}[1]{}

\newcommand{\Ot}{{\mathbf O}_{\mathrm \theta}}
\newcommand{\Or}{{\mathbf O}^{\mathrm \theta}}

\title{Association schemes in which the thin residue is an elementary abelian $p$-group of rank $2$}

\author[M.~Hirasaka]{Mitsugu Hirasaka}
\address{Department of Mathematics, Pusan National University, Busan 609-735, Republic of Korea}
\email{hirasaka@pusan.ac.kr}

\author[K.~Kim]{Kijung Kim}
\address{Department of Mathematics, Pusan National University, Busan 609-735, Republic of Korea}
\email{knukkj@pusan.ac.kr}

\date{\today}

\begin{document}
\maketitle

\begin{abstract}
In this article, we investigate the existence and schurity problem of association schemes whose
thin residues are isomorphic to an elementary abelian $p$-group of rank $2$.

\end{abstract}


\footnotetext[1]{Corresponding author: Kijung Kim}
\footnotetext[2]{The first author's research was supported by Basic Science Research Program through the National Research Foundation of Korea(NRF) funded by the Ministry of Education (2013R1A1A2012532). The second author's research was supported by Basic Science Research Program through the National Research Foundation of Korea(NRF) funded by the Ministry of Education (2013R1A1A2005349).}

\section{Introduction}\label{sec:intro}

Let $G$ be a transitive permutation group on a finite set $X$.
We denote by $\mathcal{R}_G$ the set of orbits of the induced action of $G$ on $X \times X$.
It is well known that $\mathcal{R}_G$ satisfies the following conditions:
\begin{enumerate}
\item $\{ (x, x) \mid  x \in X \} \in \mathcal{R}_G$;
\item For each $s \in \mathcal{R}_G$, $s^* := \{(x, y) \mid (y, x) \in s \} \in \mathcal{R}_G$;
\item For all $ s, t, u \in \mathcal{R}_G$ and $x, y \in X$, $| xs \cap yt^* |$ is constant whenever $(x, y) \in u$,
where $xs:=\{ y \in X \mid (x,y) \in s \}$.
\end{enumerate}

\vskip5pt
The following definition of an association scheme generalizes the above observations on the orbitals of a transitive permutation group.
Let $X$ be a non-empty set, and let $S$ be a partition of $X \times X$.
The set $S$ is called an \textit{association scheme} (or shortly  a \textit{scheme}) on $X$ if it satisfies the following conditions:
\begin{enumerate}
\item $1_X := \{ (x, x) \mid  x \in X \} \in S$;
\item For each $s \in S$, $s^* := \{(x, y) \mid (y, x) \in s \} \in S$;
\item For all $ s, t, u \in S$ and $x, y \in X$, $c_{st}^u :=| xs \cap yt^* |$ is constant whenever $(x, y) \in u$,
where $xs:=\{ y \in X \mid (x,y) \in s \}$.
\end{enumerate}

For each element $s$ in $S$, we set $n_s:=c_{ss^*}^{1_X}$ and call this (positive) integer the \textit{valency} of $s$.
For a subset $U$ of $S$, put $n_U = \sum_{u \in U} n_u$.
We call $n_S$ the \textit{order} of $S$.

\vskip5pt
We say that an association scheme $S$ is \textit{schurian} if $S=\mathcal{R}_G$ for a transitive permutation group $G$ on $X$.

\vskip5pt
Let $X$ be a set, and let $S$ be an association scheme in the above sense.
A subset $T$ of $S$ is called \textit{closed} if $\bigcup_{t \in T}t$ is an equivalence relation on $X$.
For a closed subset $T$ of $S$, we define $X/T:= \{ xT \mid x \in X\}$, called the set of \textit{cosets} of $T$ in $X$, and $S//T:=\{ s^T \mid s \in S \}$,
where $xT:= \bigcup_{t \in T} xt$ and $s^T:= \{ (xT,yT) \mid y \in xTsT \}$. Then $S//T$ is an association scheme on $X/T$ (see \cite[Theorem 4.1.3]{zies2}).
We call $S//T$ the \textit{factor} scheme of $S$ over $T$.

\vskip5pt
If $G$ acts regularly on $X$,
then $\mathcal{R}_G$ is an association scheme on $X$ such that $\mathcal{R}_G$ is a group under the relational product.
We call such an association scheme \textit{thin}.

\vskip5pt
Let $N$ be a normal subgroup of $G$ which contains a point stabilizer $G_x$.
Then $(G/N, \mathcal{R}_G)$ is a thin association scheme, where $\mathcal{R}_G$ is the set of
the orbitals of $G$ acting on the right cosets $G/N$ by the right multiplication.
An analogy of this situation gives the following concept.
For a given association scheme $S$, it is known that there exists the smallest closed subset $\Or(S)$ such that $S//\Or(S)$ is thin (see \cite{zies}).
We call $\Or(S)$ the \textit{thin residue} of $S$.

\vskip5pt
The structure of the thin residue has played an important role in the schurity problem.
It is known that association schemes are schurian if their thin residues are finite groups under the relational product
and have a distributive normal subgroup lattice (see \cite{zi}).
However, in the case that their thin residues are elementary abelian groups of rank $\geq 2$, they do not have a distributive normal subgroup lattice.
So, one might be curious what can be said in this case.
In this article, we consider the schurity problem of association schemes whose thin residues are elementary abelian groups of rank $2$.

\vskip5pt
For a closed subset $T$ of $S$, $n_S / n_T$ is called the \textit{index} of $T$ in $S$.
The following result is given in \cite[Corollary 2.8]{chk}.
Let $S$ be an association scheme such that
\begin{equation}\label{A}
\Or(S) \simeq C_p \times C_p ~~\text{and}
\end{equation}
\begin{equation}\label{B}
\{ n_s \mid s \in S \setminus \Or (S) \} = \{ p \},
\end{equation}
where $p$ is a prime number.
Then $n_S/n_{\Or(S)}  \leq p^2 + p + 1$.
We denote $n_S/n_{\Or(S)}$ by $\delta(S)$.
In this article, we investigate the existence of such association schemes under
certain extra conditions which we impose on $\delta(S)$.

\vskip5pt
Now we state main results with respect to $\delta(S)$.
\begin{itemize}
\item When $\delta(S) \leq 2$, there do not exist association schemes such that (\ref{A}) and (\ref{B}) hold (see Proposition \ref{prop:del2-ne}).
\item When $3 \leq \delta(S) \leq p+2$, we give a construction of association schemes such that (\ref{A}) and (\ref{B}) hold (see Theorem \ref{thm:main-as1}).
\item When $\delta(S) = 3$, there exists a unique association scheme such that (\ref{A}) and (\ref{B}) hold, up to isomorphism (see Theorem \ref{thm:del3-unique}).
\item When $4 \leq \delta(S) \leq p+2$, there exists a non-schurian association scheme $S$ such that (\ref{A}), (\ref{B}) and (\ref{con-three}) hold
(see Theorem \ref{thm:non-schur}).
\item When $\delta(S) = p+2$ is an odd prime, every association scheme $S$ satisfying (\ref{A}), (\ref{B}) and (\ref{con-three})
is non-schurian (see Theorem \ref{thm:main1}).
\item When $\delta(S) = p^2$ or $\delta(S) = p^2 + p + 1$, there exists a schurian association scheme such that (\ref{A}) and (\ref{B}) hold (see Section \ref{sec:main4}).
\end{itemize}

This article is organized as follows.
In Section \ref{sec:pre}, we prepare some terminology and notations.
In Section \ref{sec:main0}--\ref{sec:main4}, we prove main results.
In Section \ref{sec:con}, we give concluding remarks.

\section{Preliminaries}\label{sec:pre}
In this section, we review some notations and known facts about association schemes.
Throughout this article, we use the notation given in \cite{zies2}.

\subsection{Association schemes}\label{subsec:pre1}

Let $S$ be an association scheme on $X$.
We denote by $r(x,y)$ the element of $S$ containing $(x,y) \in X \times X$.

Let $P$ and $Q$ be nonempty subsets of $S$. We define $PQ$ to be the set of all elements $s \in S$ such that there exist elements
$p \in P$ and $q \in Q$ with $c_{pq}^s \neq 0$. The set $PQ$ is called the \textit{complex product} of $P$ and $Q$.
If one of the factors in a complex product consists of a single element $s$, then one usually writes  $s$ for $\{ s \}$.

A closed subset $T$ is called \textit{thin} if all elements of $T$ have valency 1.
The set $\{ s \mid n_s=1 \}$ is called the \textit{thin radical} of $S$ and denoted by $\Ot(S)$.
Note that $T$ is thin if and only if $T$ is a group under the relational product.

A closed subset $T$ of $S$ is called \textit{strongly normal} in $S$, denoted by $T \lhd^\sharp S$, if $s^* T s \subseteq T$ for every $s \in S$.
We put $\Or(S) := \bigcap_{T \lhd^\sharp S} T $ and call it the \textit{thin residue} of $S$.
It follows from \cite[Theorem 2.3.1]{zies} that
\begin{equation}\label{thinresidue}
\Or (S) = \langle \bigcup_{s \in S} s^*s \rangle.
\end{equation}

Let $S_1$ be an association scheme on $X_1$. A bijective map $\phi$ from $X \cup S$ to $X_1 \cup S_1$ is called an \textit{isomorphism} if it satisfies the following conditions:
\begin{enumerate}
\item $\phi(X) \subseteq X_1$ and $\phi(S) \subseteq S_1$;
\item For all $x, y \in X$ and $s \in S$ with $(x,y) \in s$, $(\phi(x), \phi(y)) \in \phi(s)$.
\end{enumerate}

An isomorphism $\phi$ from $X \cup S$ to $X \cup S$ is called an \textit{automorphism} of $S$
if $\phi(s)=s$ for all $s \in S$. We denote by $\mathrm{Aut}(S)$ the automorphism group of $S$.

On the other hand, we say that two association schemes $S$ and  $S_1$ are \textit{algebraically isomorphic} or have the \textit{same intersection numbers} if there exists a bijection $\iota$ from  $S$ to $S_1$ such that $c_{rs}^{t} = c_{\iota(r) \iota(s)} ^{\iota(t)}$ for all $r, s, t \in S$.

\subsection{Extensions of thin schemes by thin schemes}\label{subsec:pre2}

Let $G$ be a permutation group acting regularly on a finite set $X$.
Let $H$ be a normal subgroup of $G$.
We set $\tilde{g} = \{ (x, y) \in X \times X \mid x^g = y \}$.
We denote by $\widetilde{G}$ the set $\{ \tilde{g} \mid g \in G \}$.
Note that $\widetilde{G}$ is an association scheme on $X$.

\vskip5pt
For a binary relation $t \subseteq X \times X$, we define the adjacency matrix $A_t$ whose columns and rows are indexed by the elements of $X$ as follows:
\begin{equation*}\label{product}
(A_t)_{xy} = \left\{
                      \begin{array}{ll}
                      1 & \hbox{if $(x,y) \in t$;} \\
                      0 & \hbox{otherwise.}
                      \end{array}
                     \right.
\end{equation*}
Note that the condition (iii) in the second paragraph of Section \ref{sec:intro} is equivalent to $A_sA_t = \sum_{u \in S} c_{st}^uA_u$.

For $r, s \subseteq X \times X$, we set $r \cdot s = \{ (\alpha,\gamma) \mid (\alpha,\beta) \in r, (\beta,\gamma) \in s ~~\text{for some}~~ \beta \in X\}$.

\vskip5pt
Fix $x \in X$.
For $\sigma \in (G/H) \setminus \{1_{G/H} \}$, we take $t_\sigma \subseteq \bigcup_{\tau \in G/H} x^{H\tau} \times x^{H\sigma\tau}$.

In order to define an association scheme whose relation set contains
$\{ \tilde{h} \mid h \in H \} \cup \{ t_\sigma \mid \sigma \in (G/H) \setminus \{1_{G/H} \}\}$,
we consider the following conditions:

\begin{equation}\label{partition}
\{ \tilde{h} \cdot t_\sigma \mid h \in H \} = \{ t_\sigma \cdot \tilde{h} \mid h \in H \}
~~\text{is a partition of} \bigcup_{\tau \in G/H} x^{H\tau} \times x^{H\sigma\tau};
\end{equation}
\begin{equation}\label{trans}
t_\sigma^\ast = t_{\sigma^{-1}};
\end{equation}
\begin{equation}\label{product}
A_{t_\sigma} A_{t_\tau} = \left\{
                      \begin{array}{ll}
                      |K_\sigma| \sum_{s \in K_\sigma} A_s & \hbox{if $\sigma\tau = 1_{G/H}$;} \\
                      n_{\sigma, \tau} \sum_{s \in \{ t_{\sigma\tau} \cdot \tilde{h} \mid h \in H \}} A_s & \hbox{otherwise,}
                      \end{array}
                     \right.
\end{equation}
where $K_\sigma :=\{\tilde{h} \in \widetilde{H} \mid \tilde{h} \cdot t_\sigma = t_\sigma \}$
and $n_{\sigma, \tau}$ is a positive integer depending on $\sigma$ and $\tau$.
It follows from $(\ref{partition})$ that for each $\tilde{h} \cdot t_\sigma$, there exists $h_1 \in H$ such that $\tilde{h} \cdot t_\sigma = t_\sigma \cdot \tilde{h_1}$.

\begin{prop}\label{thm:conSchur}
For $\sigma \in (G/H) \setminus \{1_{G/H} \}$, let $t_\sigma$ be a subset of $\bigcup_{\tau \in G/H} x^{H\tau} \times x^{H\sigma\tau}$
satisfying $\mathrm{(\ref{partition})}$, $\mathrm{(\ref{trans})}$ and $\mathrm{(\ref{product})}$.
Put $S = \{ \tilde{h} \cdot t_\sigma \mid h \in H, \sigma \in G/H \}$, where $t_{1_{G/H}} = 1_X$.
Then $S$ is an association scheme on $X$.
\end{prop}
\begin{proof}
It follows from $(\ref{partition})$ that $S$ is a partition of $X \times X$.
Clearly, $1_X \in S$ and $\tilde{h}^\ast \in S$ for each $\tilde{h} \in \widetilde{H}$.

\vskip5pt
For each $\tilde{h} \cdot t_\sigma \in S \setminus \widetilde{H}$,
we verify that $(\tilde{h} \cdot t_\sigma)^\ast \in S$.
It suffices to check $(\tilde{h}\cdot t_\sigma)^\ast = t_{\sigma^{-1}} \cdot \tilde{h}^{-1}$ by $(\ref{partition})$.
First of all, we show $(\tilde{h}\cdot t_\sigma)^\ast \subseteq t_{\sigma^{-1}} \cdot \tilde{h}^{-1}$.
Let $(\alpha, \beta) \in (\tilde{h}\cdot t_\sigma)^\ast$.
Then
\[(\beta, \gamma) \in \tilde{h} ~~\text{and}~~ (\gamma, \alpha) \in t_\sigma\]
for some $\gamma \in X$.
If follows from definition and $(\ref{trans})$ that
$(\gamma, \beta) \in \tilde{h}^{-1} ~~\text{and}~~ (\alpha, \gamma) \in t_\sigma^\ast = t_{\sigma^{-1}}$.
So, $(\alpha, \beta) \in t_{\sigma^{-1}} \cdot \tilde{h}^{-1}$.
Thus, we have \[(\tilde{h}\cdot t_\sigma)^\ast \subseteq t_{\sigma^{-1}} \cdot \tilde{h}^{-1}.\]

Since $|(\tilde{h}\cdot t_\sigma)^\ast| = |t_{\sigma^{-1}} \cdot \tilde{h}^{-1}|$,
we have $(\tilde{h}\cdot t_\sigma)^\ast = t_{\sigma^{-1}} \cdot \tilde{h}^{-1}$ and
hence $(\tilde{h}\cdot t_\sigma)^\ast \in S$.

\vskip5pt
Finally, we show that
\[A_{\tilde{h} \cdot t_\sigma} A_{\tilde{g} \cdot t_\tau} = \sum_{s \in S}  c_{\tilde{h} \cdot t_\sigma  \tilde{g} \cdot t_\tau}^s A_s\]
for $\tilde{h} \cdot t_\sigma, \tilde{g} \cdot t_\tau \in S$.
It follows from $(\ref{partition})$ that
$(\tilde{h} \cdot t_\sigma) \cdot (\tilde{g} \cdot t_\tau) = (\tilde{h} \cdot t_\sigma) \cdot (t_\tau \cdot \tilde{g}_1)$ for some $g_1 \in H$.
Since $A_{\tilde{h} \cdot t_\sigma} = A_{\tilde{h}} A_{t_\sigma}$ and $A_{t_\tau \cdot \tilde{g}_1} = A_{t_\tau} A_{\tilde{g}_1} $,
we compute $A_{\tilde{h}}(A_{t_\sigma}A_{t_\tau})A_{\tilde{g}_1}$ by $(\ref{product})$.

If $\sigma\tau = 1_{G/H}$, then \[|K_\sigma| \sum_{s \in K_\sigma} A_{\tilde{h}} A_s A_{\tilde{g}_1} = |K_\sigma| \sum_{s \in K_\sigma} A_{\tilde{h}s\tilde{g}_1}.\]

If $\sigma\tau \neq 1_{G/H}$, then \[n_{\sigma, \tau}\sum_{s \in \{ t_{\sigma\tau} \cdot \tilde{h}_1 \mid h \in H \}} A_{\tilde{h}_1} A_s A_{\tilde{g}_1} =
n_{\sigma, \tau}\sum_{s \in \{ \tilde{h} \cdot t_{\sigma\tau} \cdot \tilde{h}_1 \cdot \tilde{g}_1 \mid h_1 \in H \}} A_s.\]

This completes the proof.
\end{proof}

\section{The case of $3 \leq \delta(S) \leq p+2$}\label{sec:main0}
First of all, we show that
there do not exist association schemes such that (\ref{A}) and (\ref{B}) hold and $\delta(S) \leq 2$.

\begin{prop}\label{prop:del2-ne}
Let $S$ be an association scheme on $X$ such that $\mathrm{(\ref{A})}$ and $\mathrm{(\ref{B})}$ hold.
Then $\delta(S) \geq 3$.
\end{prop}
\begin{proof}
If $\delta(S) = 1$, then $S = \Or(S)$.
By $\mathrm{(\ref{A})}$, $S$ is a group. By the definition of $\Or(S)$,
we have $\Or(S) = \{ 1_X \}$. This contradicts $\mathrm{(\ref{A})}$.

\vskip5pt
Suppose $\delta(S)=2$.
Then we have
\[S = \Or(S) \cup s_0\Or(S)\]
for some $s_0 \in S \setminus \Or(S)$.
Since $\Or(S)$ is thin, every element of $S \setminus \Or(S)$ has a form $s_0h$ for some $h \in \Or(S)$.
So,
\[s_0s_0^\ast  = ss^\ast\]
for each $s \in S \setminus \Or(S)$.
By $(\ref{thinresidue})$, we have $\Or(S)=s_0s_0^\ast $.
Note that $|ss^\ast |=n_s$ (see \cite[Lemma 4.2]{hirasakameta}).

\vskip5pt
The equation $(\ref{A})$ implies that every element of $S \setminus \Or(S)$ has the valency $p^2$.
This contradicts $\mathrm{(\ref{B})}$.
Therefore, we have $\delta(S) \geq 3$.
\end{proof}

In this section, we construct an association scheme $S$ such that (\ref{A}) and (\ref{B}) hold and
\begin{equation}\label{con-three}
s_1s_1^\ast \neq s_2s_2^\ast ~\text{for}~ s_1, s_2 \in S \setminus \Or(S) ~\text{with}~ s_1\Or(S) \neq s_2\Or(S).
\end{equation}
From now on, we denote by $\mathbb{F}_p$ and $G_\delta$ the finite field with $p$ elements and a finite group of order ${\delta(S)}$, respectively.

\vskip5pt
Set
\begin{equation}\label{setting}
V:= \mathbb{F}_p ^2, ~~X:= V \times G_\delta ~~\text{and}~~ G_\delta^{\times}:= G_\delta \setminus \{ 1 \},
\end{equation}
where $p$ is prime and $\delta \leq p+2$.
We denote by $u_a$ an element $(u, a)$ of $X$ for short.
We denote by $P(V)$ the set of $p+1$ subgroups of $V$ with order $p$.
We take an injective mapping $L$ from $G_\delta^{\times}$ to $P(V)$ by
\begin{equation}\label{left-mapping}
a \mapsto L_a.
\end{equation}

We denote the image of $L$ by $\{ L_a \mid a \in G_\delta^{\times}\}$.
Depending on $\{ L_a \mid a \in G_\delta^{\times} \}$,
we take a mapping $C$ from $G_\delta^{\times}$ to $P(V)$ by
\begin{equation}\label{central-mapping}
a \mapsto C_a
\end{equation}
such that $L_a \neq C_a = C_{a^{-1}}$ for $a \in G_\delta^{\times}$.

\vskip5pt
By (\ref{left-mapping}) and (\ref{central-mapping}), we choose two sets
$\{ L_a \mid a \in G_\delta^{\times} \}$ and $\{ C_a \mid a \in G_\delta^{\times} \}$.
We will construct an association scheme such that (\ref{A}), (\ref{B}) and (\ref{con-three}) hold.

\vskip5pt
Define binary relations on $X$ as follows:
\begin{equation}\label{relation1}
h_w = \{ (u_a, v_a) \mid u - v = w, a \in G_\delta \} ~~~\text{for each}~ w \in V;
\end{equation}
\begin{equation}\label{relation2}
t_a =  \left\{(u_b, v_c) \mid b^{-1}c = a, (u, v) \in \bigcup_{x \in C_a} (L_a+x)  \times (L_{a^{-1}}+x) \right\} ~~~\text{for each}~ a \in G_\delta^{\times};
\end{equation}
\begin{equation}\label{relation3}
t_1 = 1_X.
\end{equation}
Note that $h_{w_1} \cdot t_a = t_a = t_a \cdot h_{w_2}$ for $w_1 \in L_a, w_2 \in L_{a^{-1}}$.

\begin{thm}\label{thm:main-as1}
Let $\{ L_a \mid a \in G_\delta^{\times} \}$ and $\{ C_a \mid a \in G_\delta^{\times} \}$ be two sets given in $(\ref{left-mapping})$ and $(\ref{central-mapping})$.
For $w \in V, a \in G_\delta$, let $h_w, t_a$ be binary relations given in $\mathrm{(\ref{relation1})}$, $\mathrm{(\ref{relation2})}$ and $\mathrm{(\ref{relation3})}$.
Put $S =  \{t_a \cdot h_w \mid w \in V, a \in G_\delta \}$.
Then $S$ is an association scheme on $X$.
\end{thm}
\begin{proof}
By Proposition \ref{thm:conSchur},
it suffices to check (\ref{partition}), (\ref{trans}) and (\ref{product}).

\vskip5pt
\textbf{Claim 1}: For each $a \in G_\delta^{\times}$, $\left\{ (u, v) + \bigcup_{x \in C_a} (L_a +x) \times (L_{a^{-1}} +x) \mid u, v \in V \right\}$
is a partition of $V \times V$.

Since $L_a + C_a = V = L_{a^{-1}} + C_a$, we have
\begin{eqnarray*}
(u, v) + \bigcup_{x \in C_a} (L_a +x) \times (L_{a^{-1}} +x) & = &(u, 0) + (0, v) + \bigcup_{x \in C_a} (L_a + x) \times (L_{a^{-1}} + x) \\
& = & \bigcup_{x \in C_a} (L_a +x + y_1) \times (L_{a^{-1}} +x + y_2)
\end{eqnarray*}
for some $y_1, y_2 \in C_a$.


Assume
\[\bigcup_{x \in C_a} (L_a + x + y_1) \times (L_{a^{-1}} + x + y_2) \cap \bigcup_{x \in C_a} (L_a + x +y_1') \times (L_{a^{-1}} + x + y_2') \neq \emptyset.\]

Then $l_1 + x + y_1 = l_1' + x +y_1'$ and $l_2 + x + y_2 = l_2' + x +y_2'$ for some $l_1, l_1' \in L_a, l_2, l_2' \in L_{a^{-1}}$.
Since $l_1 - l_1' = y_1' - y_1 =0$ and $l_2 - l_2' = y_2' =y_2 = 0$, we have $y_1 = y_1'$ and $y_2 = y_2'$.
This completes the proof of Claim $1$.

\vskip5pt
\textbf{Claim 2}: $S$ is a partition of $X \times X$.

It is easy to see that $\{t_1 \cdot h_w \mid w \in V \}$ is a partition of $\bigcup_{a \in G_\delta} (V \times \{ a \}) \times (V \times \{ a \})$.
It suffices to verify that $\{t_a \cdot h_w \mid w \in V, a \in G_\delta^{\times} \}$ is a partition of
\[(X \times X) \setminus \bigcup_{a \in G_\delta} (V \times \{ a \}) \times (V \times \{ a \}).\]
Since \[t_a = \bigcup_{b^{-1}c=a} \left\{ (u_b, v_c) \mid (u,v) \in \bigcup_{x \in C_a} (L_a +x) \times (L_{a^{-1}} +x) \right\},\]
it follows from Claim $1$ that $\{t_a \cdot h_w \mid w \in V\}$ is a partition of
\[\{(u_b, v_c) \mid u, v \in V, b^{-1}c = a \},\]
for each $a \in G_\delta^{\times}$.
This implies that $S$ is a partition of $X \times X$.

\vskip5pt
\textbf{Claim 3}: $\{t_a \cdot h_w \mid w \in V\} = \{ h_w \cdot t_a \mid w \in V\}$ for each $a \in G_\delta^{\times}$.

Put $w = w' + w'' \in L_{a^{-1}} + C_a$. Then we have
\begin{eqnarray*}
t_a \cdot h_w & = &  \left\{(u_b, v_c) \mid b^{-1}c = a, (u, v) \in \bigcup_{x \in C_a} (L_a+x)  \times (L_{a^{-1}}+x) \right\} \cdot h_w \\
              & = &  \left\{(u_b, v_c - w_c) \mid b^{-1}c = a, (u, v) \in \bigcup_{x \in C_a} (L_a+x)  \times (L_{a^{-1}}+x) \right\} \\
              & = &  \left\{(u_b, z_c) \mid b^{-1}c = a, (u, z) \in \bigcup_{x \in C_a} (L_a+x)  \times (L_{a^{-1}}+x-w) \right\} \\
              & = &  \left\{(u_b, z_c) \mid b^{-1}c = a, (u, z) \in \bigcup_{x \in C_a} (L_a+x)  \times (L_{a^{-1}}+x-w'') \right\} \\
              & = &  \left\{(u_b, z_c) \mid b^{-1}c = a, (u, z) \in \bigcup_{x \in C_a} (L_a+ w'' +x -w'')  \times (L_{a^{-1}}+x-w'') \right\} \\
              & = &  \left\{(v_b + w''_b, z_c) \mid b^{-1}c = a, (v, z) \in \bigcup_{x \in C_a} (L_a +x -w'')  \times (L_{a^{-1}}+x-w'') \right\} \\
              & = &  h_{w''} \cdot \left\{(v_b, z_c) \mid b^{-1}c = a, (v, z) \in \bigcup_{x \in C_a} (L_a +x)  \times (L_{a^{-1}}+x) \right\} = h_{w''} \cdot t_a.
\end{eqnarray*}

\vskip5pt
\textbf{Claim 4}: $t_a^\ast = t_{a^{-1}}$ for each $a \in G_\delta^{\times}$.

By the definition of $t_a$, we have
\begin{eqnarray*}
t_a^\ast & = & \left\{(v_c, u_b) \mid b^{-1}c = a, (u, v) \in \bigcup_{x \in C_a} (L_a +x) \times (L_{a^{-1}} +x) \right\}  \\
         & = & \left\{(v_c, u_b) \mid c^{-1}b = a^{-1}, (v, u) \in \bigcup_{x \in C_a} (L_{a^{-1}} +x) \times (L_a +x) \right\}.
\end{eqnarray*}
Since $C_a = C_{a^{-1}}$, we have $t_a^\ast = t_{a^{-1}}$.

\vskip5pt
\textbf{Claim 5}: $A_{t_a} A_{t_{a^{-1}}} = p \sum_{w \in L_a} A_{h_w}$ for each $a \in G_\delta^{\times}$.

Since $t_a^\ast = t_{a^{-1}}$,
we have
\[t_a \cdot t_{a^{-1}} = \left\{(u_b, w_b) \mid  b \in G_\delta,  (u, w) \in \bigcup_{x \in C_a} (L_a +x)  \times (L_a +x) \right\}.\]
Thus, $t_a \cdot t_{a^{-1}} = \bigcup_{w \in L_a} h_w$.
For each $(u_b, w_b) \in \bigcup_{w \in L_a} h_w$, we have
\[u_b t_a \cap w_b t_a = \{ v_{ba} \mid v \in L_a^{-1} + x \}\]
for some $x \in C_a$.
This completes the proof of Claim $5$.

\vskip5pt
\textbf{Claim 6}: $A_{t_a} A_{t_d} =  \sum_{w \in L_{ad} } A_{t_{ad} \cdot h_w}$ if $a^{-1} \neq d$ for $a, d \in G_\delta^{\times}$.

Put $t_d = \left\{ (v_c, w_e) \mid c^{-1}e = d, (v, w) \in \bigcup_{x \in C_d} (L_d+x)  \times (L_{d^{-1}}+x) \right\}$.
Clearly, $t_a \cdot t_d \subseteq t_{ad} \cdot \widetilde{L}_{ad}$.
For each $(u_b, w_e) \in t_{ad} \cdot \widetilde{L}_{ad}$, we have $|u_b t_a \cap w_e t_d^\ast|=1$,
since $|(L_{a^{-1}}+x) \cap (L_d+y)|=1$ for each $x \in C_a, y \in C_d$.

\vskip5pt
The above claims complete the proof.
\end{proof}


Now we show that
there exists a unique association scheme such that (\ref{A}) and (\ref{B}) hold and $\delta(S)=3$, up to isomorphism.

\begin{thm}\label{thm:del3-unique}
Let $S_i$ be an association scheme on $X_i$ $(i=1,2)$ such that $\mathrm{(\ref{A})}$ and $\mathrm{(\ref{B})}$ hold and $\delta(S) = 3$.
Then $S_1$ is isomorphic to $S_2$.
\end{thm}
\begin{proof}
For each $s_k \in S_k \setminus \Or(S_k)$ $(k = 1, 2)$,
it follows from showing $t \in s_k^\ast s_k \Leftrightarrow s_kt=s_k$ that the set $s_k^\ast s_k$ is a subgroup of $\Or(S_k)$
(see \cite[Lemma 4.2]{hirasakameta}).
Since $\{s_k^\ast s_k \mid s_k \in S_k \setminus \Or(S_k)\}$ has exactly two elements as the subgroups $\Or(S_k)$,
we denote them by $\langle u_k \rangle$ and $\langle v_k \rangle$, respectively.
Note that there exists a unique subgroup $\langle c_k \rangle$ such that $c_k s_k = s_k c_k$ for each $s_k \in S_k$.

\vskip5pt
First of all, we choose $s_k, t_k \in S_k \setminus \Or(S_k)$ such that
\[s_k^\ast s_k = \langle u_k \rangle, ~~t_k^\ast t_k = \langle v_k \rangle \leq \Or(S_k)  ~~\text{and}~~  s_k^\ast = t_k.\]

We denote by $X_{k,l}$ an element of $X_k/\Or(S_k)$ $(l= 1, 2, 3)$.
Since $s_k^\ast \in s_k^2$, we have $1_{X_k} \in s_k^3$.
So, we can fix
\[x_l \in X_{1,l}, ~~y_l \in X_{2,l} ~(l= 1, 2, 3)~~\text{such that}\]
\[x_2 \in x_1 s_1, ~~x_3 \in x_2 s_1, ~~x_1 \in x_3 s_1 ~~\text{and}~~ y_2 \in y_1 s_2, ~~y_3 \in y_2 s_2, ~~y_1 \in y_3 s_2.\]

\vskip5pt
It is known that the automorphism group of $C_p \times C_p$ is isomorphic to $GL_2(p)$.
Note that $PGL_2(p)$ is sharply $3$-transitive of degree $p+1$ (see \cite{dixon}).
Thus, there exists an group isomorphism $\Phi$ from $\Or(S_1)$ to $\Or(S_2)$ such that $\Phi(\langle u_1 \rangle) = \langle u_2 \rangle$,
$\Phi(\langle v_1 \rangle) = \langle v_2 \rangle$ and $\Phi(\langle c_1 \rangle) = \langle c_2 \rangle$.

By replacing the generators if necessary, without loss of generality, we assume $\Phi(u_1) = u_2$, $\Phi(v_1) = v_2$ and $\Phi(c_1) = c_2$.

\vskip5pt
Define a bijective map $\phi$ from $X_1 \cup S_1$ to $X_2 \cup S_2$ as follows:
\[X_{1,l} \rightarrow X_{2,l} ~~\text{by}~~ x_lu_1^iv_1^j  \mapsto y_lu_2^iv_2^j  ~~(l = 1, 2, 3);\]
\[\Or(S_1) \rightarrow \Or(S_2) ~~\text{by}~~ u_1^iv_1^j \mapsto u_2^iv_2^j;\]
\[S_1 \setminus \Or(S_1) \rightarrow S_2 \setminus \Or(S_2) ~~\text{by}~~  s_1v_1^j \mapsto s_2v_2^j, u_1^jt_1 \mapsto u_2^jt_2.\]

\vskip10pt
Now we check that $\phi$ is an isomorphism.
For all $\alpha, \beta \in X_1$ and $h \in S_1$ with $(\alpha, \beta) \in h$, it suffices to show $(\phi(\alpha), \phi(\beta)) \in \phi(h)$.
First of all, we divide our consideration into two cases, i.e., $h \in \Or(S_1)$ and $h \in S_1 \setminus \Or(S_1)$.

\vskip5pt
In the case of $h \in \Or(S_1)$, without loss of generality, we assume that $(\alpha, \beta) \in h \cap (X_{1,1} \times X_{1,1})$.
Then
\[(x_1, \alpha) \in h_1 ~~\text{and}~~(x_1, \beta) \in h_2\]
for some $h_1, h_2 \in \Or(S_1)$.
This means $(\alpha, \beta) \in h_1^\ast h_2 = h$.
Since $(y_1, \phi(\alpha)) \in \phi(h_1)$ and $(y_1, \phi(\beta)) \in \phi(h_2)$,
we have
\[(\phi(\alpha), \phi(\beta)) \in \phi(h_1)^\ast \phi(h_2)  = \phi(h_1^\ast h_2) = \phi(h).\]

\vskip5pt
In the case of $h \in S_1 \setminus \Or(S_1)$, first of all, we assume that $(\alpha, \beta) \in h \cap (X_{1,1} \times X_{1,2})$.
Then
\[ (x_1,x_2) \in s_1, (x_1, \alpha) \in u_1^i v_1^j ~~\text{and}~~ (x_1, \beta) \in u_1^l v_1^m \]
for some $0 \leq i, j, l, m \leq p-1$.
This means $(\alpha, \beta) \in u_1^{p-i} v_1^{p-j} s_1 u_1^l v_1^m$.
Note that $s_k s_k^\ast s_k = s_k$ and $t_k^\ast t_k s_k = s_k$ $(k = 1, 2)$.
So, we have
\[u_1^{p-i} v_1^{p-j} s_1 u_1^l v_1^m = u_1^{p-i} s_1 v_1^m = s_1 v_1^{m+e}\]
for some $e$.
Since $(y_1, y_2) \in s_2$, $(y_1, \phi(\alpha)) \in u_2^i v_2^j$ and $(y_1, \phi(\beta)) \in u_2^l v_2^m$,
we have
\[(\phi(\alpha), \phi(\beta)) \in u_2^{p-i} v_2^{p-j} s_2 u_2^l v_2^m = u_2^{p-i} s_2 v_2^m.\]
The assumption that $\Phi(u_1) = u_2$, $\Phi(v_1) = v_2$ and $\Phi(c_1) = c_2$ implies $u_2^{p-i} s_2 v_2^m = s_2 v_2^{m+e}$.
Since $\phi(s_1 v_1^{m+e}) = s_2 v_2^{m+e}$, we have $(\phi(\alpha), \phi(\beta)) \in \phi(h)$.
\vskip5pt
We also assume that $(\alpha, \beta) \in h \cap (X_{1,2} \times X_{1,3})$ or $(\alpha, \beta) \in h \cap (X_{1,3} \times X_{1,1})$.
By the same argument, we have $(\phi(\alpha), \phi(\beta)) \in \phi(h)$.
\vskip5pt
The symmetry of $s_k$ and $t_k$ completes the proof.
\end{proof}


\begin{thm}\label{thm:main1}
Let $S$ be an association scheme on $X$ such that $\mathrm{(\ref{A})}$, $\mathrm{(\ref{B})}$ and $\mathrm{(\ref{con-three})}$ hold and $\delta(S) = p+2$ is an odd prime.
Then it is non-schurian.
\end{thm}
\begin{proof}
Suppose that $S$ is schurian.
For convenience, we denote $\mathrm{Aut}(S)$ by $G$.

\vskip5pt
\textbf{Claim 1}: The order of $G$ is $p^3(p+2)$.

By the orbit-stabilizer property (see \cite[page 8]{dixon}), we have
$|G| =  |G_\alpha| |X|$ for a fixed $\alpha \in X$.
Since $|X|=p^2(p+2)$, it suffices to verify $|G_\alpha|=p$.

Let $r_1$ is an element of $S \setminus \Ot(S)$.
We fix an element $\beta$ of $X$ such that $\beta \in \alpha r_1$.
Under the assumption that $(X, S)$ is schurian,
we have
\[|G_\alpha| = |G_{\alpha, \beta}||\alpha r_1| = |G_{\alpha, \beta}|p\]
by the orbit-stabilizer property.
We shall show $|G_{\alpha, \beta}|=1$.

Let $\gamma$ be an element of $\alpha r_1$ such that $\gamma \in \beta t$ for some $t \in R(r_1) \setminus \{1_X\}$.
$G_{\alpha, \beta}$ fixes $\gamma$ since $c_{r_1 t}^{r_1} = 1$. Thus, $G_{\alpha, \beta}$ fixes each element of $\alpha r_1$.

We consider an arbitrary element $r_2 \in S \setminus \Ot(S)$ such that $r_1 \neq r_2$.
For $\delta \in \alpha r_2$, we assume that there exists $g \in G_{\alpha, \beta}$ such that $\delta^g \neq \delta$.
Then there exist $s_1, s_2 \in S$ such that $(\beta, \delta), (\beta, \delta^g) \in s_1$ and $(\gamma, \delta), (\gamma, \delta^g) \in s_2$.
This means $c_{s_1 s_2^\ast}^t \geq 2$ and $c_{r_2 s_1^\ast}^{r_1} \geq 2$.

We divide our consideration into the following cases.
\begin{enumerate}
\item If $r_1 \Or(S) = r_2 \Or(S)$, then $s_1, s_2 \in \Ot(S)$.
This contradicts $c_{s_1 s_2^\ast}^t = 1$.
Thus, $G_{\alpha, \beta}$ fixes each element of $\alpha r_2$. This implies that $G_{\alpha, \beta}$ fixes each element of $\alpha s$ with $r_1 \Or(S) = s \Or(S)$.
\item If $r_1 \Or(S) \neq r_2 \Or(S)$, then $s_1, s_2 \in S \setminus \Ot(S)$.
This contradicts $c_{r_2 s_1^\ast}^{r_1} =1$.
\end{enumerate}
Hence, we have $|G_{\alpha, \beta}|=1$.

\vskip5pt
\textbf{Claim 2}: The center $Z(G)$ is not trivial.

Let $P$ and $Q$ be a Sylow $p$-subgroup and a Sylow $(p+2)$-subgroup of $G$, respectively.
Since $p+2$ is prime,
we have $P \unlhd G$ by the Sylow theorem (the number $n_p(G)$ of Sylow $p$-subgroups of $G$ is a divisor of $|G|$, and $n_p(G) \equiv 1$ (mod $p$)).

\vskip5pt
If $P$ is non-abelian, then $|Z(P)|=p$.
Since $Z(P)$ is characteristic in $P$, we have $Z(P) \unlhd G$.
By N/C theorem (see \cite[page 41]{isaacs}),
\[G/C_G(Z(P)) = N_G(Z(P))/C_G(Z(P)) \simeq ~\text{a subgroup of}~ \mathrm{Aut}(Z(P)).\]
Since $\mathrm{Aut}(Z(P))$ is cyclic of order $p-1$, we have $|G/C_G(Z(P))|=1$ and hence $Z(P) \leq Z(G)$.

If $P$ is abelian, then, by N/C theorem,
\[G/C_G(P) = N_G(P)/C_G(P) \simeq ~\text{a subgroup of}~ \mathrm{Aut}(P).\]
Since $|\mathrm{Aut}(P)|$ is not divided by $p+2$, we have $|G/C_G(P)|=1$ and hence $P=Z(G)$.

\vskip15pt
\textbf{Claim 3}: There exists a thin closed subset $C(S)$ such that
$n_{C(S)} = p$ and $cs =sc$ for all $s \in S, c \in C(S)$.

By Claim 2, there exists $z \in Z(G)$ such that $\alpha \neq \alpha^z$ for some $\alpha \in X$.
Since $zg = gz$ for each $g \in G_\alpha$, we have $(\alpha^z)^{G_\alpha} = \alpha^z$.
This implies that $r(\alpha, \alpha^z)$ is a thin element of $S$.

Since $G$ is transitive on $X$, each element $s$ of $S$ can be represented by $r(\alpha, \alpha^g)$ for some $g \in G$.
Since $r(\alpha, \alpha^z)$ is thin, $r(\alpha, \alpha^z)r(\alpha, \alpha^g)$ and $r(\alpha, \alpha^g)r(\alpha, \alpha^z)$ are singletons.
We check the followings:
\begin{enumerate}
\item $r(\alpha, \alpha^z) r(\alpha, \alpha^g) = r(\alpha, \alpha^z) r(\alpha^z, \alpha^{gz}) = r(\alpha, \alpha^{gz})$;
\item $r(\alpha, \alpha^g) r(\alpha, \alpha^z) = r(\alpha, \alpha^g) r(\alpha^g, \alpha^{zg}) = r(\alpha, \alpha^{zg})$.
\end{enumerate}
Since $zg = gz$, we have $r(\alpha, \alpha^{gz}) = r(\alpha, \alpha^{zg})$.

\vskip5pt
On the other hand, since $\delta(S) = p+2$, $\{ ss^\ast \mid s \in S \setminus \Or(S) \}$ consists of $p+1$ nontrivial subgroups of $\Or(S)$.
So, there exists an element $s \in S$ such that $cs = s$ and $sc \neq s$ for all $c \in C(S)$.
This contradicts Claim $3$.
\end{proof}


By a parallel argument of \cite[Section 4]{kim1},
we can construt a non-schurian association scheme which is algebraically isomorphic to a given schurian association scheme.

\begin{prop}\label{schurian-non}
Suppose that $S$ is a schurian association scheme satisfying $(\ref{A})$, $(\ref{B})$, $(\ref{con-three})$ and $S//\Or(S) \cong C_n$ for some
$4 \leq n \leq p+1$.
Then there exists a non-schurian association scheme with the same intersection numbers as $S$.
\end{prop}

\begin{rem}
We omit the proof of Proposition \ref{schurian-non} and mention the following.
In \cite[Section 4]{kim1}, it is considered a schurian association scheme satisfying $(\ref{A})$, $(\ref{B})$, $(\ref{con-three})$ and $\delta(S) = p \geq 5$.
To follow the idea of \cite[Section 4]{kim1}, it suffices to consider the condition that $S//\Or(S)$ is cyclic, and
the proof of \cite[Theorem 1.2]{kim1} uses the fact that $\delta(S)$ is at least $4$.
\end{rem}

\begin{thm}\label{thm:non-schur}
Let $p$ be an odd prime. Then for each $4 \leq n \leq p+2$,
there exists a non-schurian association scheme satisfying $(\ref{A})$, $(\ref{B})$, $(\ref{con-three})$ and $\delta(S)=n$.
\end{thm}
\begin{proof}
For each $4 \leq n \leq p+2$, it follows from Theorem \ref{thm:main-as1} that there exists an association scheme $S$ satisfying the above conditions.

When $n=p+2$, we are done by Theorem \ref{thm:main1}.

When $4 \leq n \leq p+1$, if $S$ is non-schurian, then we are done.
If not, then we construct a non-schurian association scheme by Proposition \ref{schurian-non}.
\end{proof}

\section{The cases of $\delta(S) = p^2$ and $\delta(S) = p^2+p+1$}\label{sec:main4}
For the schurian scheme $\mathcal{R}_G$ given in Section \ref{sec:intro},
there is a one-to-one correspondence between $X$ and the set $G/G_x$ of right cosets of a point stabilizer $G_x$.
Also there is a one-to-one correspondence between $\mathcal{R}_G$ and the set of double cosets of $G_x$ in $G$.

\vskip5pt
In the following subsections, for a given group $G$ and its subgroup $H$, we consider the action on the right cosets $G/H$ by the right multiplication.

\vskip5pt
\subsection{$\delta(S) = p^2$}

Let $P := \{ A(a,b) \mid a, b \in \mathbb{F}_p \}$, where $A(a,b)$ is a $3 \times 3$ matrix defined by
\begin{equation}\label{C}
\begin{bmatrix}
1 &  a  & b   \\
0 &  1  & a   \\
0 &  0  & 1
\end{bmatrix}.
\end{equation}

Define a subgroup of $\mathrm{AGL}_3 (\mathbb{F}_p)$ as follows:
\[G := \{ t_{A, x} \mid A \in P, x \in \mathbb{F}_p^3 \},\]
where $t_{A, x} : y \mapsto yA + x$ is a mapping from $\mathbb{F}_p^3$ to $\mathbb{F}_p^3$.

Put $H := \{ t_{I, x} \mid x \in \{ (c, 0, 0) \mid c \in \mathbb{F}_p \} \}$.
We show that the action of $G$ on $G/H \times G/H$ induces an association scheme of order $p^4$ such that (\ref{A}) and (\ref{B}) hold.

\vskip10pt
\textbf{Claim 1}: $G \unrhd N_G(H)=\{t_{I, x} \mid x \in \mathbb{F}_p^3 \}$.

First of all, we verify $N_G(H)=\{t_{I, x} \mid x \in \mathbb{F}_p^3 \}$.
Since $t_{A, z}^{-1} : y \mapsto yA^{-1} - zA^{-1}$,
we have $t_{A, z} t_{I, x} t_{A, z}^{-1} : y \mapsto y + xA$.
For $t_{I, x} \in H, t_{A, z} \in G$, we have
\[t_{A, z} t_{I, x} t_{A, z}^{-1} \in H  ~~\text{if and only if}~~  xA=(a', a'a, a'b) \in \{ (c, 0, 0) \mid c \in \mathbb{F}_p \},\]
where $x = (a', 0, 0)$ and $A$ has the form of (\ref{C}).
Thus, $t_{A, z} \in N_G(H)$ if and only if $a=b=0$.
It is easy to see that $G \unrhd N_G(H)$.

\vskip10pt
\textbf{Claim 2}: For each $t_{B, z} \in G \setminus N_G(H)$,  $H t_{B, z} H$ is the union of $p$ right cosets in $G/H$.

First of all, we take an element $t_{I, x_1} t_{B, z} t_{I, x_2}$ in $H t_{B, z} H$, where $x_1 = (a_1, 0, 0)$ and $x_2 = (a_2, 0, 0)$.
Since $t_{I, x_1} t_{B, z} t_{I, x_2} : y \mapsto ((yI+x_2)B + z)I + x_1 = yB + x_2B +z +x_1$,
this implies
\[t_{I, x_1} t_{B, z} t_{I, x_2} \in H \{ t_{B, w} \mid w \in \{ x_2 B + z \mid x_2 = (a_2, 0, 0), a_2 \in  \mathbb{F}_p \} \}.\]
Thus, we have
\[H t_{B, z} H = H \{ t_{B, w} \mid w \in \{ x_2 B + z \mid x_2 = (a_2, 0, 0), a_2 \in  \mathbb{F}_p \} \}.\]

\vskip10pt
\textbf{Claim 3}: The right stabilizer $St_R(t_{B, z})$ of $H t_{B, z} H$ is
$H \{ t_{I, x} \mid x \in \langle(\delta,0,0)B^{-1}\rangle \}$ for some $\delta \in \mathbb{F}_p^\times$.

For $t_{I, x_1} t_{B, z} t_{I, x_2} \in H t_{B, z} H, ~~t_{A, z_1} \in G$,
since $t_{I, x_1} t_{B, z} t_{I, x_2} t_{A, z_1} : y \mapsto yAB + (z_1 + x_2)B + z + x_1$,
we have
\[t_{I, x_1} t_{B, z} t_{I, x_2} t_{A, z_1} \in H t_{B, z} H\]
\[\text{if and only if}\]
\[A=I ~\text{and}~ z_1 \in  \{(\delta, 0, 0)B^{-1} + (c, 0, 0) \mid c \in \mathbb{F}_p \} ~\text{for some}~ \delta \in \mathbb{F}_p^\times.\]

\vskip5pt
From Claim $1$, we have that $N_G(H)/H \simeq C_p \times C_p$ and the action of $G$ on $G/N_G(H) \times G/N_G(H)$ induces a thin association scheme.
From Claim $3$, we have that $St_R(t_{B, z}) \cup St_R(t_{B^{-1}, z})$ generates $N_G(H)$.
So, (\ref{A}) holds.
It follows from Claim $2$ that (\ref{B}) holds.

\vskip5pt
\subsection{$\delta(S) = p^2+p+1$}

\vskip10pt
Let $K$ be a subgroup of $\mathbb{F}_{p^3}^\times$ with $|K|=p^2 + p + 1$, where $p$ is an odd prime.

Define
\[G:= \{t_{a,b} \mid a \in K, b \in \mathbb{F}_{p^3} \},\]
where $t_{a, b} : y \mapsto ay + b$ is a mapping from $\mathbb{F}_{p^3}$ to $\mathbb{F}_{p^3}$.

Put $H:= \{t_{1,c} \mid c \in \mathbb{F}_p \}$.
Note that $N_G(H) \cap \{ t_{a, 0} \mid a \in K \} = \{ t_{1, 0} \}$ since $|\mathbb{F}_p^\times|$ does not divide $|K|$.
We show that the action of $G$ on $G/H \times G/H$ induces an association scheme of order $p^2(p^2 + p+ 1)$ such that (\ref{A}) and (\ref{B}) hold.

\vskip10pt
\textbf{Claim 1}: $G \unrhd N_G(H)=\{ t_{1,d} \mid d \in \mathbb{F}_{p^3} \}$.

First of all, we verify $N_G(H)=\{t_{1, d} \mid d \in \mathbb{F}_{p^3} \}$.
Since $t_{a, b}^{-1} : y \mapsto a^{-1}y - a^{-1}b$,
we have $t_{a,b} t_{1,c} t_{a,b}^{-1} : y \mapsto y + ac$.
For $t_{1, c} \in H, t_{a,b} \in G$, we have
\[t_{a,b} t_{1,c} t_{a,b}^{-1} \in H  ~~\text{if and only if}~~  a=1.\]
This implies that $N_G(H)=\{t_{1, d} \mid d \in \mathbb{F}_{p^3} \}$.
It is easy to see that $G \unrhd N_G(H)$.

\vskip10pt
\textbf{Claim 2}: For each $t_{a, b} \in G \setminus N_G(H)$,  $H t_{a, b} H$ is the union of $p$ right cosets in $G/H$.

First of all, we take an element $t_{1, c_1} t_{a, b} t_{1, c_2}$ in $H t_{a, b} H$.
Since $t_{1, c_1} t_{a, b} t_{1, c_2} : y \mapsto a(y + c_2) + b + c_1 = ay + ac_2 + b + c_1$,
this implies $t_{1, c_1} t_{a, b} t_{1, c_2} \in H  \{ t_{a, z} \mid z \in \{ ac_2 + b \mid c_2 \in \mathbb{F}_p \} \}$.
Thus, we have $H t_{a, b} H = H  \{ t_{a, z} \mid z \in \{ ac_2 + b \mid c_2 \in \mathbb{F}_p \} \}$.

\vskip10pt
\textbf{Claim 3}: The right stabilizer $St_R(t_{a, b})$ of $H t_{a, b} H$ is $H\{ t_{1, c} \mid c \in a^{-1}\mathbb{F}_p \}$.

For $t_{1, c_1} t_{a, b} t_{1, c_2} \in H t_{a, b} H, ~~t_{d, e} \in G$,
since $t_{1, c_1} t_{a, b} t_{1, c_2} t_{d, e} : y \mapsto ady + a(e + c_2) + b + c_1$,
we have
\[t_{1, c_1} t_{a, b} t_{1, c_2} t_{d, e} \in H t_{a, b} H
~~\text{if and only if}~~ d=1 ~\text{and}~ e \in a^{-1}\mathbb{F}_p + \mathbb{F}_p.\]

\vskip5pt
From Claim $1$ and $3$, we have that $N_G(H)/H \simeq C_p \times C_p$ and the action of $G$ on $G/N_G(H) \times G/N_G(H)$ induces a thin association scheme.
From Claim $3$, we have that $St_R(t_{a, b}) \cup St_R(t_{a^{-1}, b})$ generates $N_G(H)$.
So, (\ref{A}) holds.
It follows from Claim $2$ that (\ref{B}) holds.

\section{Concluding Remarks}\label{sec:con}

In \cite[Section 4]{mp}, it was shown that any reduced Klein configuration defines an incidence structure.
By a parallel argument, it was obtained a partial linear space over a scheme under certain assumption (see \cite{chk}).

\vskip5pt
In geometric terms, a \textit{partial linear space} is a pair $(\mathcal{P},\mathcal{L})$, where $\mathcal{P}$ is a set of points and $\mathcal{L}$ is a family of subsets of $\mathcal{P}$, called \textit{lines} such that every pair of points is contained in at most one line and every line contains at least two points. Moreover, $(\mathcal{P}, \mathcal{L})$ is called a \textit{linear space} if every pair of points is contained in exactly one line.

\vskip5pt
Let $S$ be an association scheme on $X$ with $T:= \Or(S) \cong C_p \times C_p$.
Let $\{ X_i \mid i=1, 2, \dots, m  \}$ denote the set of cosets of T in $S$.
Then we define
\[L(s):=\{ t \in T \mid ts=s \}.\]

For given $i, j$ and $s \in S$ with $s \cap (X_i\times X_j) \neq \emptyset$, we define $L_{ij}(s):= L(s)$.
Since $L_{ij}(s) = L_{ij}(s')$ for $s' \in TsT$, we write $L_{ij}$ for $L_{ij}(s)$.

\vskip5pt
Define an incidence structure $(\mathcal{P}, \mathcal{L})$ as follows:
\begin{enumerate}
\item $\mathcal{P}:= \{ i \mid i= 1, 2, \dots, m  \}$;
\item $\mathcal{L}:= \{ L_i(M) \mid \{1_X\} < M < T, |L_i(M)|\neq 1, 1 \leq i \leq m \}$, where $L_i(M):= \{ i \} \cup \{ j \mid L_{ij}=M \}$.
\end{enumerate}
An \textit{automorphism} $\tau$ of $(\mathcal{P}, \mathcal{L})$ is a permutation of $\mathcal{P}$ such that $\mathcal{L}^\tau = \mathcal{L}$.
We denote by $\mathrm{Aut}(\mathcal{P}, \mathcal{L})$ the set of automorphisms of $(\mathcal{P}, \mathcal{L})$.

\vskip5pt
Then $(\mathcal{P}, \mathcal{L})$ is a partial linear space in which any point belongs to at most $p+1$ lines
and $\mathrm{Aut}(S//T)$ is a subgroup of $\mathrm{Aut}(\mathcal{P}, \mathcal{L})$ (see \cite{chk}).

\begin{thm}\label{thm:space}
Let $(\mathcal{P}, \mathcal{L})$ be a linear space.
Suppose that $\mathrm{Aut}(\mathcal{P}, \mathcal{L})$ has a regular subgroup $G$ on $\mathcal{P}$
and $|\{l \in \mathcal{L} \mid x \in l \}| < p+1$, where $x \in \mathcal{P}$.
Then the linear space $(\mathcal{P}, \mathcal{L})$ gives a construction of an association scheme $S$ such that $(\ref{A})$ and $(\ref{B})$ hold and
\begin{enumerate}
\item $S//\Or(S) \simeq G$,
\item the incidence structure arising from $S$ is isomorphic to $(\mathcal{P}, \mathcal{L})$.
\end{enumerate}
\end{thm}
\begin{proof}

We identify $\mathcal{P}$ with $G$. In particular, we assume that $x$ is the identity of $\mathcal{P}$.
For a fixed $x$, we define $N= \{l \in \mathcal{L} \mid x \in l \}$.
Set $V:= \mathbb{F}_p^2$, $X:= V \times \mathcal{P}$ and $\mathcal{P}^{\times}:= \mathcal{P} \setminus \{ x \}$.
We denote by $u_a$ an element $(u, a)$ of $X$ for short.
We denote by $\{ L_l \mid l \in N \}$ the set of $|N|$ distinct subgroups of order $p$ in $V$.
Fix a subgroup $C$ of $V$ such that $C \not\in \{ L_l \mid l \in N \}$.

\vskip10pt
Define binary relations on $X$ as follows:
\[ h_w = \{(u_a, v_a) \mid u - v = w, a \in \mathcal{P} \} ~~~\text{for each}~ w \in V; \]
\[ t_x = 1_X; \]
\[ t_a =  \left\{(u_b, v_c) \mid b^{-1}c  = a, (u, v) \in \bigcup_{y \in C} (L_l+y)  \times (L_m+y) \right\} ~~~\text{for each}~ a \in \mathcal{P}^{\times},\]
where $L_l$ and $L_{m}$ are elements of $\{ L_n \mid n \in N \}$ such that $a \in l$ and $m= \{ da^{-1} \mid d \in l \}$.

Put \[S :=  \{t_a \cdot h_w \mid w \in V, a \in \mathcal{P} \}.\]

\vskip10pt
Now we show that $S$ is an association scheme on $X$.
Clearly, $1_X \in S$ and $h_w^\ast = h_{-w}$ for each $w \in V$.
It follows from the same argument of Claim $1$ and $2$ of Theorem \ref{thm:main-as1} that $S$ is a partition of $X \times X$.

\vskip5pt
\textbf{Claim 1}: $t_a^\ast = t_{a^{-1}}$ for each $a \in \mathcal{P}^{\times}$.

By the definition of $t_a$, we have
\begin{eqnarray*}
t_a^\ast & = & \left\{(v_c, u_b) \mid b^{-1}c = a, (u, v) \in \bigcup_{y \in C} (L_l +y) \times (L_m +y)\right\}  \\
                                  & = & \left\{(v_c, u_b) \mid c^{-1}b = a^{-1}, (v, u) \in \bigcup_{y \in C} (L_m +y) \times (L_l +y) \right\}.
\end{eqnarray*}
Since $l = \{ ea \mid e \in m \}$, we have $t_a^\ast = t_{a^{-1}}$.

\vskip5pt
\textbf{Claim 2}: $A_{t_a} A_{t_{a^{-1}}} = p \sum_{w \in L_l} A_{h_w}$ for each $a \in \mathcal{P}^{\times}$, where $a \in l$.

Since $t_a^\ast = t_{a^{-1}}$,
we have
\[t_a \cdot t_{a^{-1}} = \left\{(u_b, w_b) \mid  b \in F_q,  (u, w) \in \bigcup_{y \in C} (L_l +y)  \times (L_l +y) \right\}.\]
Thus, $t_a \cdot t_{a^{-1}} = \bigcup_{w \in L_l} h_w$.
For $(u_b, w_b) \in \bigcup_{w \in L_l} h_w$, we have
\[u_b t_a \cap w_b t_a = \{ v_{ba} \mid v \in L_m + y \}\]
for some $y \in C$,
where $m = \{ da^{-1} \mid d \in l \}$.
This completes the proof of Claim 2.

\vskip5pt
\textbf{Claim 3}: $A_{t_a} A_{t_d} =  p A_{t_{ad}}$ if $l_1 = l_2$ for $a^{-1} \in l_1, d \in l_2$.

Put $t_a =  \{(u_b, v_c) \mid b^{-1}c  = a, (u, v) \in \bigcup_{y \in C} (L_l+y)  \times (L_{l_1}+y) \}$ and
$t_d =  \{(u'_c, v'_e) \mid c^{-1}e  = d, (u', v') \in \bigcup_{y \in C} (L_{l_2}+y)  \times (L_m+y) \}$.
Since $l_1 = l_2$,
we have
\[t_a \cdot t_d = \left\{(u_b, w_e) \mid  b^{-1}e  = ad,  (u, w) \in \bigcup_{y \in C} (L_l +y)  \times (L_m +y) \right\}.\]

\vskip5pt
\textbf{Claim 4}: $A_{t_a} A_{t_d} =  \sum_{w \in L_l } A_{t_{ad} \cdot h_w}$ if $l_1 \neq l_2$ for $a^{-1} \in l_1, d \in l_2$,
where $ad \in l \in N$.

Put $t_a =  \{(u_b, v_c) \mid b^{-1}c  = a, (u, v) \in \bigcup_{y \in C} (L_l+y)  \times (L_{l_1}+y) \}$ and
$t_d =  \{(u'_c, v'_e) \mid c^{-1}e  = d, (u', v') \in \bigcup_{y \in C} (L_{l_2}+y)  \times (L_m+y) \}$.
Clearly, $t_a \cdot t_d \subseteq t_{ad} \cdot \widetilde{L}_{f}$, where $ad \in f \in N$.
For each $(u_b, w_e) \in t_{ad} \cdot \widetilde{L}_m$, we have $|u_b t_a \cap w_e t_d^\ast|=1$,
since $|(L_{l_1}+ y) \cap (L_{l_2}+ y)|=1$ for $y \in C$.

\vskip5pt
For $t_a \cdot h_w, t_b \cdot h_v \in S$, we have $(t_a \cdot h_w) \cdot (t_b \cdot h_v) = (h_u \cdot t_a) \cdot (t_b \cdot h_v)$ for some $u \in V$.
The above claims complete the proof.
\end{proof}

\vskip 10pt
When $\Or(S) \cong C_2 \times C_2$, based on the computational result of \cite{hanakimi}
we observe the existence of schemes with respect to $\delta(S)$.
In Table \ref{hm}, the symbols $\mathrm{E}$ and $\mathrm{N}$ mean the existence and non-existence, respectively.

\begin{table}[h]
\begin{center}
\begin{tabular}{|c||c|c|c|c|c|}\hline
$\delta(S)$ & 3&4&5&6&7 \\ \hline
$|X|$ & 12&16&20&24&28 \\ \hline
$(X,S)$ &E&E& N& N& E \\ \hline
\end{tabular}
\caption{The case of $\Or(S) \cong C_2 \times C_2$}\label{hm}
\end{center}
\end{table}

\begin{itemize}
\item It follows from Theorem \ref{thm:del3-unique} that
the scheme corresponding to the case of $\delta(S)=3$ of Table \ref{hm} is \texttt{as-12.no.49} in \cite{hanakimi}.

\end{itemize}

\vskip 10pt
When $\Or(S) \cong C_3 \times C_3$, based on the computational result of \cite{hanakimi}
we observe the existence of schemes with respect to $\delta(S)$.
In Table \ref{hm3}, the symbol $?$ means that the existence is not yet determined.

\begin{table}[h]
\begin{center}
\begin{tabular}{|c||c|c|c|c|c|c|c|c|c|c|c|}\hline
$\delta(S)$ & 3&4&5&6&7 &8&9&10&11&12&13 \\ \hline
$|X|$ & 27&36&45&54&63 &72&81&90&99&108&117 \\ \hline
$(X,S)$ &E& E& E& $?$&  $?$ &$?$ &E& $?$& $?$& $?$&E \\ \hline
\end{tabular}
\caption{The case of $\Or(S) \cong C_3 \times C_3$}\label{hm3}
\end{center}
\end{table}

\begin{itemize}
\item It follows from Theorem \ref{thm:main1} that
the non-schurian scheme corresponding to the case of $\delta(S)=5$ of Table \ref{hm} is \texttt{as-45.class-20.no.5} in \cite{hanakimi}.
\item In the cases of $\delta(S) =9, 13$, Section \ref{sec:main4} guarantees the existence of such schemes.
\end{itemize}

\vskip 10pt
Finally, we close our article by giving further work.
\begin{itemize}
\item For $\delta(S) \geq p+3$, we need to consider the existence of association schemes. Based on Theorem \ref{thm:space} we guess that there is a connection between linear spaces and association schemes.
\item In the case of $\delta(S) = p^2$ or $\delta(S) =p^2 +p +1$, we gave a construction of schurian association schemes.
It seem to be an important problem to find out a method which is analogous to Proposition \ref{schurian-non}.
\end{itemize}

\vskip15pt
\textbf{Acknowledgement}

The authors would like to thank anonymous referees for their valuable comments.

\bibstyle{plain}

\end{document}